\documentclass[12pt]{article}
\usepackage[english]{babel}
\usepackage[cp1251]{inputenc}
\usepackage{amsmath,amssymb,graphicx,amsthm,enumerate,comment,amscd}
\usepackage[unicode,pdftex]{hyperref}
\usepackage[matrix,arrow,curve]{xy}
\graphicspath{{kartinki/}}

\newtheorem{theorem}{Theorem}

\newtheorem{corollary}{Corollary}
\newtheorem{lemma}{Lemma}
\newtheorem{proposition}{Proposition}
\newtheorem{claim}{Claim}

\theoremstyle{remark}
\newtheorem{remark}{Remark}

\theoremstyle{definition}

\newcommand{\Z}{\mathbb{Z}}

\newcommand{\del}{\smash{\mskip3mu\lower1truept\hbox{\vdots}\mskip3mu}}

\renewcommand{\phi}{\varphi}

\renewcommand{\kappa}{\varkappa}

\newcommand{\sk}{\mathrm{sk}}
\renewcommand{\int}{\mathrm{int}}

\newcommand{\Deg}{\mathrm{Deg}\,}

\newcommand{\TT}{\mathcal{T}}

\newcommand{\RP}{\mathbb{RP}}

\begin{document}

\title{Maximal degree of a map of surfaces}
\author{Andrey Ryabichev%
\footnote{\ IUM; Moscow, Russia //
\texttt{ryabichev@179.ru}
}}
\date{}

\maketitle

\begin{abstract}
Given closed possibly nonorientable surfaces $M,N$, 
we prove that if a map $f:M\to N$ has degree $d>0$, then $\chi(M)\le d\cdot\chi(N)$.
We give all necessary comments on the definition and properties of geometric degree,
which can be defined for any map.
Our proof is based on the Kneser-Edmonds factorization theorem,
simple natural proof of which is also presented.

\end{abstract}

\section{Introduction}

Through this paper, we set $M,N$ to be closed connected surfaces, possibly nonorientable.

Given a map $f:M\to N$,
we define its {\it degree $\Deg f$} as a minimal cardinality of the preimage of a regular value
among all smooth maps in the homotopy class of $f$.
Note that if $M$ and $N$ are orientable, then
this number coincides with the usual notion of degree,
we discuss features of the definition and prove some of its properties in \S\ref{s:degree}.

Our goal is to prove the following fact:

\begin{theorem}\label{th:degree}
Let 
$f:M\to N$ be a map of degree $d>0$.
Then $\chi(M)\le d\cdot\chi(N)$.
\end{theorem}

This fact is well known,
apparently it was first proved by Kneser \cite{kneser-30} in the case of orientable surfaces.
It looks very similar to the assertions ``{\it all maps $S^2\to N$ are nullhomotopic}''
and ``{\it if $\chi(M)>\chi(N)$, then $d=0$}''.
These statements can be easily proved
using universal covering or the intersection form in cohomology.
However, similar elementary approaches to Theorem~\ref{th:degree} are not known to the author.

The possible ways to prove Theorem~\ref{th:degree} are rather a bit more technical.
For orientable surfaces, one can use
the Milnor-Wood inequality \cite[Th.~1.1]{wood}
or the Gromov norm \cite[\S5.35]{gromov}.
In this paper, we present the most elementary proof
including the following factorization theorem of Edmonds.
First, let us recall the notation.

\vspace{.5em}

Suppose we have a $2$-submanifold with boundary $K\subset M$,
such that for every component $K_i\subset K$ its boundary $\partial K_i$ is connected.
Collapsing each $K_i$ we obtain a smooth manifold $Q$.
The factorization map $p:M\to Q$ is called {\it a pinch map}.

A map of closed surfaces $q:Q\to N$ is called {\it a branched covering},
if for some discrete subset $B\subset N$ the restriction of $q$ to $q^{-1}(B)$ is a covering.

\begin{theorem}\label{th:factirisation}
Every map 
$f:M\to N$ is
either homotopic to the composition of a pinch map $p:M\to Q$ with a branched covering $q:Q\to N$ for some closed surface $Q$,
or homotopic to a map whose image is a graph imbedded into $N$.
\end{theorem}

This theorem is due to Edmonds \cite{edmonds}.
He also holds the case of surfaces with boundary,
considering maps $f$ which restriction to the boundary is $(\Deg f)$-sheeted covering.
His proof was corrected and improved by Skora \cite{skora}.

We present another proof of Theorem~\ref{th:factirisation} which is simpler and more natural in some ways.
Namely, it does not use induction and construct the factorization in one step.
The idea is to take a triangulation of $N$ and consider a map $h$ homotopic to $f$ which is transversal to $\sk^1(N)$
and has minimal number of edges of $h^{-1}(\sk^1(N))$,
and then to deform $h$ over each triangle.
This approach was inspired by the Lurie's proof of Dehn-Nielsen theorem \cite[Lect.~38]{lurie},
see also \cite[\S8.3.1]{farb-margalit}.

Unlike \cite{edmonds} and \cite{skora},
in our proof we do not control a degree of a branched covering.
Also our approach does not deal with surfaces with boundary,
the corresponding generalization is possible but it would require some additional work.
Finally, we do not use the theory of absolute degree,
see \S\ref{s:degree} for some remarks on this subject,
so we tried to make our reasoning completely self-contained.

\subsection*{Acknowledgments}
I wish to express my gratitude to Ian Agol,
who drew my attention to the Edmond's paper~\cite{edmonds} answering my Mathoverflow question,
and also to Allan Edmonds, who paid attension to this discussion and gave a link to the Skora's article \cite{skora}.
I am very thankful to participants of the Moscow seminar ``Geometric walks'',
in particular to Alex Gorin, for a fruitful discussion of the approaches to Theorem~\ref{th:degree}.
I want to thank Sashachka Pilipyuk for some critical remarks.
Finally, I am grateful to the Independent University of Moscow and the Young Russian Mathematics award for a financial support.

\section{Preliminaries}

\subsection{Conventions and notation on surfaces and transversality}

We use the term {\it surface} for a $2$-manifold, and {\it closed} with respect to a manifold means that it is compact without boundary.
All manifolds will be assumed to be infinitely-smooth, as well as maps.
The maps we construct sometimes will be not smooth and should be smoothed if needed,
this inaccuracy will not cause difficulties, see e.\,g.\ \cite[\S8]{hirsch-book}.

Every open subset of a surface $U\subset M$ that we take
is supposed to be ``sufficiently nice'' --- namely, it should be an interior of a compact $2$-submanifold with boundary.
We refer to this boundary as $\partial U$ (although formally it not belongs to $U$).

When we {\it cut} a closed surface $M$ {\it along a closed curve $C\subset M$},
we assume to obtain a compact surface with boundary $M'$ as a result,
so the points of $C$ will double in $M'$.


Given surfaces $M,N$, we say that a map $f:M\to N$ is {\it transversal}
to a stratified subset $\mathcal{S}\subset N$,
if $f$ is transversal to every its stratum.
Namely, every vertex $y\in\mathcal{S}$ must be a regular value of $f$
and for every edge $C\subset\mathcal{S}$ and any $x\in f^{-1}(C)$
there must be a vector $v\in T_x M$ such that $df(v)\notin T_{f(x)}C$.

By the implicit function theorem, $f^{-1}(\mathcal{S})\subset M$ is a stratified subset.
If $\mathcal{S}$ is closed, then the set of maps transversal to $\mathcal{S}$
is a dense and open in $C^\infty(M,N)$.
See e.\,g.\ \cite[Part~1, \S1]{goresky-macpherson} for details.

\subsection{Geometric degree}\label{s:degree}

Recall, we define a {\it degree} of a map $f:M\to N$
as a minimal $d\in\Z_{\le0}$ such that there are
a~smooth map $h:M\to N$ homotopic to $f$ and
its regular value $y\in N$ such that $|h^{-1}(y)|=d$.
It is known as {\it geometric degree}, but in further sections we will call it just a degree
and denote as $\Deg f$.

The degree theory began with the work of Hopf \cite{hopf-28}, \cite{hopf-30}
and was developed with Olum \cite{olum} and Epstein \cite{epstein}.
The most important properties of geometric degree in dimension $2$ was proved by Kneser \cite{kneser-28}, \cite{kneser-30}.

Here we will state and sketch the proofs of a few properties of the degree.
We consider not the famous properties but only which will be used in \S\ref{s:degree-th-proof} 
in order to prove Theorem~\ref{th:degree}.
For a more detailed review see e.\,g.\ \cite{brown-schirmer} or \cite{sklyarenko} in addition to \cite{epstein} and \cite{olum}.
As usual, we suppose $M,N$ to be closed connected surfaces,
but one can similarly formulate and prove corresponding statements
for any closed manifolds of the same dimension.

\subsubsection{Degree of a branched covering}

Suppose a map $f:M\to N$ is {\it orientation-true}.
This means that it maps orientation-preserving/reversing loops in $M$
to orientation-preserving/reversing loops in $N$ respectively.
Clearly, this is equivalent to the equality $f^*(w_1(N))=w_1(M)$ for Stiefel-Whitney classes,
or to the fact $f^*(\Z_N)\simeq\Z_M$, where $\Z_M$ denotes the orientation local system of $M$ with fiber $\Z$.
Then $f$ induces a homomorphism $H^2(N;\Z_N)\to H^2(M;\Z_M)$, where both groups here are isomorphic to $\Z$.
Its coefficient is called the {\it cohomological degree} of $f$, denote it by $\deg f$.
Note that for non-orientation-true maps one can similarly define a cohomological degree
only as a residue $\pmod2$
because of $H^2(M;f^*(\Z_N))\simeq\Z_2$.
For more details on local systems see e.\,g.\ \cite{spanier} or \cite[Ch.\,VI]{whitehead}.

\begin{proposition}\label{pr:orientation-true}
If the map $f:M\to N$ is orientation-true, then $\deg f=\Deg f$.
\end{proposition}


\begin{proof}
Deform $f$ so that a regular value $y\in N$ has $\Deg f$ preimages.
Choose a local orientation of $M$ and $N$.
That allows us to define {\it the sign} for every preimage $x_i\in f^{-1}(y)$,
so that the sum equals $\deg f$
(formally, here we use local cohomology).
This implies that $\deg f\le\Deg f$.

Assume that the inequality is strict,
so some preimages of $y$ have different signs.
Then we can try to cancel them by a homotopy of $f$
and obtain a contradiction with $|f^{-1}(y)|=\Deg f$.
For this purpose we take a $CW$-decomposition of $N$ with $\sk^0(N)=y$
and suppose that $f$ is transversal to $\sk^1(N)$.
The preimage $f^{-1}(\sk^1(N))$  is a graph.
Since $M$ is connected, all preimages of $y$
contain in one component of the graph, call it $\Gamma$.
By the assumption, there is an edge $e\subset\Gamma$
which image does not cover the whole of corresponding $1$-cell of $N$
(otherwise all preimages of $y$ have the same sign). 
Then the ends of $e$ can be cancelled.
This argument is very similar to the proof of Claim~\ref{c:no-loops} in \S\ref{s:properties} below.
\end{proof}


\begin{corollary}\label{col:covering-degree}
For a $k$-sheeted branched covering $f:M\to N$ we have $\Deg f=k$.
\end{corollary}

\begin{proof}
Clearly, $f$ is orientation-true.
So the required directly follows by Proposition~\ref{pr:orientation-true} and the fact that $\deg f=k$.
Note that here we only use the inequality $\deg f\le\Deg f$ 
from the simple part of the proof of Proposition~\ref{pr:orientation-true}.
\end{proof}

\subsubsection{Degree of the composition of maps}

\begin{proposition}\label{pr:composition-covering}
Suppose $f:M\to N$ is a map.
Let $q:N'\to N$ be a $k$-sheeted covering such that $f$ lifts to $N'$,
e.\,g. there is $f':M\to N'$ such that $f=q\circ f'$.
Also suppose $N'$ is connected. 
Then $\Deg f=k\cdot\Deg f'$.
\end{proposition}

\begin{proof}
Homotope $f$ so that the regular value $y\in N$ has $\Deg f$ preimages.
This homotopy can be lifted to a homotopy of $f'$. 
So we obtain that every $y'\in q^{-1}(y)$ is a regular value of $f'$, 
and this immediately implies that $\Deg f\ge k\cdot\Deg f'$.

The opposite inequality is evident, since we can homotope $f'$ so that
{\it any} discrete subset of $N'$ consists of regular values with $\Deg f'$ preimages.
\end{proof}

In fact, the last argument shows that $\Deg (f_2\circ f_1)\le\Deg f_1\cdot\Deg f_2$
for any maps $f_1,f_2$.
The opposite inequality in Proposition~\ref{pr:composition-covering}
is nontrivial 
(and does not hold if $q$ is a branched covering)
because of the following pathology.

\begin{remark}\label{r:branchedcovering-pinch}
Let $f_1:S^2\to\RP^2$ be the universal covering, then $\Deg f_1=2$ by Corollary~\ref{col:covering-degree}.
Let $f_2:\RP^2\to S^2$ be a map collapsing one projective line $l\subset\RP^2$ to a point, 
then $\Deg f_2=1$ since it induces a nonzero homomorphism in $H^2(-;\Z_2)$ and therefore it is non-nullhomotopic.
However, the composition $f_2\circ f_2$ is nullhomotopic: 
is maps both hemispheres of the domain $S^2$ surjectively to the range~$S^2$, but with different orientations.
So $\Deg(f_2\circ f_1)=0$.
\end{remark}

\begin{remark}\label{r:pinch-branchedcovering}
Let $f_2:\RP^2\to S^2$ be as above, we have $\Deg f_2=1$.
Let $f_3:S^2\to S^2$ be any map of degree $7$ as an element of $\pi_2(S^2)$,
then $\Deg f_3=7$ according to Corollary~\ref{col:covering-degree}.
However, $\Deg(f_3\circ f_2)=1$, and moreover, $f_3\circ f_2\sim f_2$.
As one can see, there are exactly two homotopy classes of maps $\RP^2\to S^2$,
since the obstruction for such maps to be homotopic lies in $H^2(\RP^2;\pi_2(S^2))\simeq\Z_2$ 
(see e.\,g.~\cite[Ch.\,VI,~\S6]{whitehead}).
One also can directly construct a homotopy of $f_3\circ f_2$ to $f_2$.
So we have another example when $\Deg (f_3\circ f_2)\ne\Deg f_3\cdot\Deg f_2$.
Note that here $f_2$ is homotopic to a pinch map,
and as $f_3$ we can take a branched covering.
\end{remark}

\section{The factorization theorem}

\subsection{A map with a minimal graph}

\begin{proof}[Proof of Theorem~\ref{th:factirisation}]
Take some triangulation of $N$ and denote its $1$-skeleton by $\TT\subset N$.
Consider maps $h:M\to N$ homotopic to $f$ which are transversal to $\TT$.
Then $h^{-1}(\TT)$ is an imbedded graph in $M$, call it $\Gamma$, possibly with {\it isolated circles}
whose images do not cover the vertices of~$\TT$.

We take $h$ such that $\Gamma$ has a minimal number of edges $E(\Gamma)$.
Each isolated circle is counted as one edge.
Then we observe that $\Gamma$ have the following three properties which we will prove in \S\ref{s:properties}.

\begin{claim}\label{c:no-loops}
For every edge of $\Gamma$ the images of its endpoints do not coincide.
\end{claim}

So, for every component $A\subset M\setminus\Gamma$
its image $h(A)$ is contained in certain triangle $B\subset N$ and $h(\partial A)\subset\partial B$.
Moreover, by Claim~\ref{c:no-loops}, we may assume that
every component $\alpha\subset\partial A$ is either an isolated circle,
or $\alpha$ maps to $\partial B$ monotonously with index $i_\alpha\ne0$.
In the last case we call $\alpha$ {\it essential}.

\begin{claim}\label{c:no-circles}
Either $\Gamma$ has no isolated circles, or $\Gamma$ is a union of such circles and has no vertices.
\end{claim}

In case when $\Gamma$ is a union of isolated circles,
the image of $h$ contains in $N\setminus\sk^0(\TT)$.
Note that $N\setminus\sk^0(\TT)$ can be deformation retracted onto the dual graph of $\TT$.
This proves the theorem in that case.

Further we will assume that $\Gamma$ has no isolated circles.
Take a triangle $B\subset N$ and  a component $A\subset h^{-1}(B)$.
Orient $\partial B$ ant $\partial A$ so that $h$ preserves the orientation.

\begin{claim}\label{c:boundary-orientation}
Either $\partial A$ consists of one component with index $1$,
or $A$ is orientable and
all the components of~$\partial A$ has the same orientation with respect to $A$.
\end{claim}

If $A$ is nonorientable, then $h|_A$ is homotopic to a pinch map 
so that the homotopy is stationary in a neighborhood of $\partial A$.

Otherwise, orient $A$.
We can present $A$ as a connected sum of certain number of disks $D_1,\ldots,D_k$
and a closed surface $S$ (which is possibly a sphere).
We may assume that they are joined by cylinders $C_1,\ldots,C_k$ in that order.

Homotope $h$ so that each $C_j$ maps to a single point as well as $S$.
Then on every $D_j$ we can homotope $h$ to a branched covering with one critical point of index $i_{\partial D_j}$
or to a diffeomorphism if $i_{\partial D_j}=1$.
The homotopy assumed to be stationary in a neighborhood of $\partial A$ and the resulting $h$ preserves the orientation.
Finally, homotope $h$ near 
$C_1,\ldots,C_{k-1}$
to a branched covering with two branched point of index $2$.
The subsurface $C_k\cup S\subset A$ remains pinched.

Repeat this for all components $A\subset M\setminus\Gamma$,
and the proof of Theorem~\ref{th:factirisation} is complete.
\end{proof}

\subsection{Properties of the minimal graph}\label{s:properties}

\begin{proof}[Proof of Claim~\ref{c:no-loops}]
Recall that $h$ maps vertices of $\Gamma$ to vertices of $\TT$
and interiors of edges of~$\Gamma$ to interiors of edges of $\TT$.
By transversality of $h$ to the vertices of $\TT$,
the half-edges of any vertex $v\in\Gamma$
are in bijective correspondence with the half-edges of $h(v)\in\TT$.
Therefore, since $\TT$ has no loops, then $\Gamma$ has no loops either.

Take an edge $e\subset\Gamma$ with endpoints $v$ and $w$.
Suppose $h(v)=h(w)$.
Let $e'\subset\TT$ be a (closed) edge which contains $h(e)$.
Take a small tubular neighborhood $U\supset e'$
and a tubular neighborhood $V\supset e$ such that $h(V)\subset U$
and $h(\partial V)\subset\partial U$.

\begin{figure}[h]
\center{\includegraphics[width=100mm]{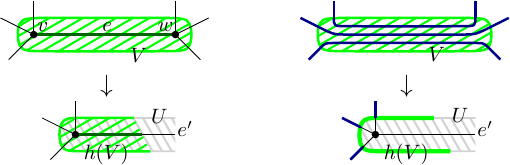}}
\caption{Collapsing of an edge with the same images of the endpoints}\label{fig:collapsing-edge}
\end{figure}

Since $h(v)=h(w)$, 
the image $h(V)$ do not cover the whole of~$e'$.
Then we can homotope $h$ on $V$ and squeeze $h(V)$ outside $U$, so that $h(V)\subset\partial U$.
We will see that the edge $e$ disappeared and other edges incident on $v$ and $w$ will modify as in Figure~\ref{fig:collapsing-edge}.

Thus, $E(\Gamma)$ decreased at least by $1$,
that contradicts the minimality.
\end{proof}

\begin{proof}[Proof of Claim~\ref{c:no-circles}]
Suppose $\Gamma$ has isolated circles and has vertices as well.
Then find a component $A\subset M\setminus\Gamma$
whose boundary includes at least one isolated circle $\alpha_0$ and at least one essential curve $\alpha_1$.
To do this, one can consider the dual graph of $\Gamma$,
mark blue the edges dual to isolated circles,
mark red the edges dual to essential curves,
and then find a vertex incident on edges with different colors.

Take a triangle $B\subset N$ such that $h(A)\subset B$.
Take points $x_0\in\alpha_0$ and $x_1\in\alpha_1$ such that $h(x_0)=h(x_1)=y\in\partial B$.
Take a non self-intersecting path $\gamma$ in $A$ from $x_0$ to $x_1$.
Then $h\circ\gamma$ is a loop inside the triangle $B$ with basepoint $y\in\partial B$.

\begin{figure}[h]
\center{\includegraphics[width=130mm]{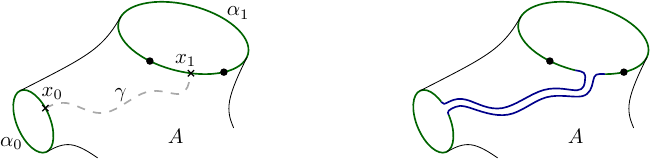}}
\caption{Join of an isoleted circle with an essential curve}\label{fig:deliting-circle}
\end{figure}

Similarly to the proof of Claim~\ref{c:no-loops},
we can homotope $h$ to compress $h(\gamma)$ to $y$ and then to squeeze $h(\gamma)$ outside $B$,
so that the homotopy is stationary outside a small neighborhood of~$\gamma$.
As a result, $\alpha_0$ and $\alpha_1$ will join together in one curve, see Figure~\ref{fig:deliting-circle}.
Thus $E(\Gamma)$ decreased by $1$,
that contradicts the minimality.
\end{proof}

\begin{proof}[Proof of Claim~\ref{c:boundary-orientation}]
Suppose the hypothesis of the claim is violated.
Then one can find two points $x_1,x_2\in\partial A$ such that $h(x_1)=h(x_2)=y\in N$
and a non self-intersecting curve $\gamma$ in $A$ from $x_1$ to $x_2$
which admits a coorientation agreed with the orientation of $\partial A$ at $x_1$ and $x_2$.
(Indeed, if $A$ is nonorientable, we can take any $x_1,x_2$ with the same image and then choose $\gamma$ properly.
And if $A$ is orientable, we take $x_1$ and $x_2$ on the components of $\partial A$ with different orientation with respect to $A$ and take any~$\gamma$.)

Then we can homotope $h$ in a small neighborhood of $\gamma$ similarly to the proof of Claim~\ref{c:no-circles}.
Suppose $x_1$ belongs to the oriented edge $a_1b_1$ of $\Gamma$, and $x_2$ belongs to the oriented edge $a_2b_2$.
Then after the homotopy $\Gamma$ modifies as in Figure~\ref{fig:deliting-circle}
so that these edges will be replaced by the edges $a_1a_2$ and $b_1b_2$.

Note that the homotopy will not change $E(\Gamma)$.
But also note that $h(a_1)=h(a_2)$ and $h(b_1)=h(b_2)$,
that contradicts Claim~\ref{c:no-loops}.
\end{proof}

\begin{remark}
Clearly, if $f$ is the composition of a pinch map with a branched covering, then $\Deg f\ne0$.
In case $\Deg f=0$ one can strengthen Theorem~\ref{th:factirisation} as
{\it $f$ is homotopic to the composition of a retraction of $M$ to a graph $\Gamma'\subset M$
with a projection $\Gamma'\to N$}.
Note that as $\Gamma'$ we can take the dual graph of $\Gamma$ above,
but it may be not isomorphic to the dual graph of $\TT$.
\end{remark}

\section{Estimation of the degree}\label{s:degree-th-proof}

Note that in Theorem~\ref{th:factirisation}
the pinched subsurface of $M$ may be assumed to be connected (or empty).
Also, in order to prove Theorem~\ref{th:degree}, we note the following remark.

\begin{proposition}\label{pr:nonorientable-branch}
The resulting map in Theorem~\ref{th:factirisation}
cannot both pinch nonorientable subsurface of $M$ and have a branch point.
\end{proposition}

Note that this assertion refers to the decomposition obtained
just in our proof of Theorem~\ref{th:factirisation}.
Of course one can compose a pinch of a crosscap with a branched covering
(e.\,g.\ as in Remark~\ref{r:pinch-branchedcovering}),
but this is not our case.

\begin{proof}
Suppose that the obtained pinch map $p:M\to Q$
collapses a M\"{o}bius band $L\subset M$
(and, possibly, some other subsurface).
Note that we can ``move'' $L$ across $M$.
Namely, we can replace $L$ by a point
and for any component $A\subset M\setminus\Gamma$ glue a M\"{o}bius band instead of any point inside $A$
defining $p$ on it 
as a collapse to a single point.

As such $A$ we take a component of $M\setminus\Gamma$ with a disconnected boundary
or a component whose boundary has index $>1$.
Otherwise, if there is no such a component,
the resulting map will have no branched points,
as we can see from the final part of the proof of Theorem~\ref{th:factirisation},
and our proposition holds.

Then, after the ``moving'' $L$ into the $A$,
the statement of Claim~\ref{c:boundary-orientation} does not hold the obtained map.
But the moving of $L$ does not change $\Gamma$,
that contradicts the minimality of $E(\Gamma)$.
\end{proof}

\begin{proposition}\label{pr:deg-q-deg-f}
For the factorization $M\xrightarrow{p}Q\xrightarrow{q}N$ from our proof of Theorem~\ref{th:factirisation},
if $q$ is a $d$-fold branched covering, then $\Deg f=d$.
\end{proposition}

Since $p$ may be not orientation-true,
this assertion is nontrivial in view of Remark~\ref{r:pinch-branchedcovering}.

\begin{proof}
Consider three cases.
{\it If both $M$ and $N$ are orientable,}
then $p$ is always orientation-true.
It remains to apply Proposition~\ref{pr:orientation-true}
and the functoriality of cohomology.

{\it Suppose $M$ is orientable and $N$ is nonorientable.}
Then $f$ lifts to the orientation double cover $\widetilde N\to N$.
Indeed, the map $q$ is orientation-true, 
so is the composition $q\circ p\sim f$,
this allows us to lift $f|_{\sk^1(M)}$ to the cover $\widetilde N\to N$,
and the lifting on $2$-cells of $M$ exists and unique by the homotopy lifting property.
So we can apply Proposition~\ref{pr:composition-covering}, Corollary~\ref{col:covering-degree}
and the proved case of orientable $M,N$.

{\it Suppose $M$ is nonorientable.}
By Proposition~\ref{pr:nonorientable-branch} 
either $p$ is orientation-true, or $q$ is a covering.
In the first case the composition $q\circ p\sim f$ is orientation-true as well, 
and it remains to apply Proposition~\ref{pr:orientation-true}.
In the second case we just apply Propositions~\ref{pr:composition-covering}, 
Corollary~\ref{col:covering-degree}
and the fact $\Deg p=1$, which holds since $p$ induces a nonzero homomorphism in $H^2(-;\Z_2)$.
\end{proof}

\begin{proof}[Proof of Theorem~\ref{th:degree}]
Apply Theorem~\ref{th:factirisation}.
Denote the set (possibly empty) of critical values of $q$ by $B\subset N$.
By Proposition~\ref{pr:deg-q-deg-f} the number of sheets of $q$ equals $d$.
Then 
$$\chi(Q\setminus q^{-1}(B))=d\cdot\chi(N\setminus B),$$
therefore $\chi(Q)\le d\cdot\chi(N)$
(this is similar to the reasoning in the Riemann-Hurwitz formula).
To complete our proof, note that $\chi(M)\le\chi(Q)$.
\end{proof}

\section{Appendix: stable maps and apparent contours}

Now let us show an unexpected application
of the factorization theorem.

A map of surfaces $f:M\to N$ is called {\it stable},
if its critical points $\Sigma(f)$ are folds and cusps
and the set of critical values $f(\Sigma(f))$ has only transversal self-crossings, called {\it nodes}
(for more details see e.\,g.\ \cite{yamamoto}, see also \cite{ryabichev-surfaces}).
Denote the number of nodes by $n(f)$.

Note that stable maps forms dense open subset in $C^\infty(M,N)$.
Pinches and branched coverings, which we considered above, are not stable,
but one can deform them to stable maps by an arbitrary small homotopy.

\begin{theorem}
Every map of closed surfaces $f:M\to N$ is homotopic to $f'$ such that $n(f')=0$.
\end{theorem}

This fact connects the study of apparent contours of T.\,Yamamoto
with the theory of branched coverings,
see e.\,g.\ \cite{gabai-kases}.

\begin{proof}
Apply Theorem~\ref{th:factirisation} to the map $f$.
Then one can homotope $q$ so that each branch point of index $i$
will turn to a fold circle with $i+2$ cusps.
A collapsing of an orientable handle can be turned to a fold curve with $4$ cusps,
see Fig.~\ref{fig:collapsing-handle}.

\begin{figure}[h]
\center{\includegraphics[width=60mm]{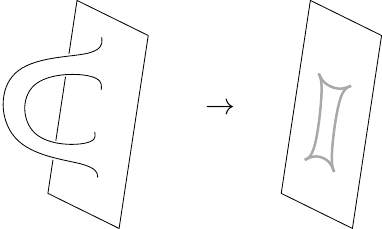}}
\caption{Collapsing of a handle adds a fold curve with $4$ cusps}\label{fig:collapsing-handle}
\end{figure}

Finally, a collapsing of a M\"{o}bius band is homotopic 
to a map with fold along the M\"{o}bius band
and with a curve with one cusp around it,
see Fig.~\ref{fig:collapsing-crosscap}
(on the left, the opposite points on the inner circle should be identified, 
we show them connected by dotted lines).

\begin{figure}[h]
\center{\includegraphics[width=80mm]{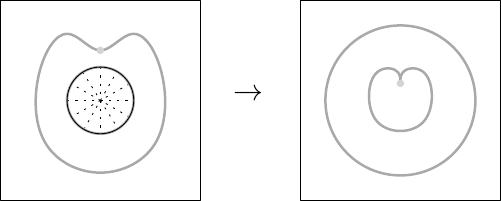}}
\caption{Collapsing of a M\"{o}bius band via adding two fold curves with $1$ cusp}\label{fig:collapsing-crosscap}
\end{figure}

The described homotopies are local,
the images of these curves are small and we may assume that they do not cross each other.
\end{proof}


\end{document}